\documentclass[12pt,a4paper]{amsart}
\numberwithin{equation}{section}
\usepackage{amssymb}
\usepackage{amsmath}
\newtheorem{lemma}{Lemma}
\newtheorem{thm}{Theorem}
\newtheorem*{thrs}{Theorem (R\"uckert, Schleicher)}
\newtheorem*{cor}{Corollary}
\theoremstyle{remark}

\newtheorem*{ack}{Acknowledgment}
\def\R{\mathbb{R}}
\def\N{\mathbb{N}}
\def\C{\mathbb{C}}
\def\CC{\widehat{\mathbb{C}}}
\def\bl{\begin{lemma}}
\def\el{\end{lemma}}

\def\dist{\operatorname{dist}}
\begin{document}
\title[Newton's method]{Newton's method and Baker domains}
\author{Walter Bergweiler}
\thanks{Supported by the G.I.F.,
the German--Israeli Foundation for Scientific Research and
Development, Grant G-809-234.6/2003, 
the EU Research Training Network CODY,
and the ESF Research Networking Programme HCAA}
\address{Mathematisches Seminar,
Christian--Albrechts--Universit\"at zu Kiel,
Lude\-wig--Meyn--Str.~4,
D--24098 Kiel,
Germany}
\email{bergweiler@math.uni-kiel.de}
\begin{abstract}
We show that there exists an entire function $f$ without zeros
for which the associated Newton function $N_f(z)=z-f(z)/f'(z)$ is
a transcendental meromorphic functions without Baker domains.
We also show that there exists an entire function $f$ with exactly
one zero for which the complement of the immediate attracting basin
has at least two components and contains no 
invariant Baker domains of~$N_f$. The second  result answers a question of
J.~R\"uckert and D.~Schleicher while the first one gives a partial
answer to a question of X.~Buff.
\end{abstract}
\maketitle
\section{Introduction and results} \label{intro}
Newton's method
for finding the zeros of an entire $f$ consists of
iterating the meromorphic function
$$N_f(z):=z-\frac{f(z)}{f'(z)},$$
see~\cite{Ber93} for 
an introduction to the iteration theory of meromorphic functions, 
including a section on Newton's method.
If $\xi$ is a zero of~$f$, then 
there exists an
$N_f$-invariant component $U$ of the Fatou set of $N_f$ containing $\xi$
in which the iterates $N_f^k$ of $N_f$ converge to $\xi$ as $k\to\infty$.
This domain $U$ is called the {\em immediate basin} of~$\xi$.

There may also be $N_f$-invariant components of the Fatou set of $N_f$
in which the iterates of $N_f$ tend to~$\infty$.  
We call such an $N_f$-invariant domain a {\em virtual immediate basin}.
(This is in slight deviation from~\cite{BufRuc,RucSch} where the definition
of a virtual immediate basin additionally includes the existence of an 
``absorbing set''; cf.\ the remark in~\S 3.3.)
It was suggested by Douady that the existence of virtual immediate basins
is related to $0$ being an asymptotic value of~$f$.
This relationship was investigated in~\cite{BerDraLan,BufRuc}.
While it was shown in~\cite{BerDraLan} that in general the
existence of  a virtual immediate basin does not imply that $0$ is an
asymptotic value of~$f$, this conclusion was shown to be true under
suitable additional hypotheses in~\cite{BufRuc}.
It was also shown in~\cite{BufRuc} that if $f$ has a logarithmic 
singularity over~$0$, then $N_f$ has a virtual immediate basin.

If $f$ has the form $f=Pe^Q$ where $P$ and $Q$ are polynomials,
with $Q$ nonconstant, then the Newton function
$N_f$ is rational, $\infty$ is
a parabolic fixed point  of $N_f$ and the
associated parabolic domains are virtual immediate basins.
If $f$ does not have the above form, then $N_f$ is transcendental.
An invariant Fatou component where the iterates of $N_f$ tend
to $\infty$ is then called an {\em invariant Baker domain}.
So except in the case where $f=Pe^Q$ a virtual immediate basin
is an invariant Baker domain.

If $f$ has no zeros, then $f$ has the asymptotic value $0$ by 
Iversen's theorem~\cite[p.~292]{Ne}. This suggests that there could always
be virtual immediate basins if there are no zeros.
We show that this is not the case in general.
\begin{thm}
\label{th1}
There exists an entire function $f$ without zeros for which $N_f$ is 
a transcendental meromorphic function without invariant Baker domains.
\end{thm}
The following corollary is obvious.
\begin{cor}
\label{co1}
There exists 
a transcendental meromorphic function without fixed points and
without invariant Baker domains.
\end{cor}
This is a partial answer to a question of Buff who had asked whether
there exists a transcendental  
{\em entire} function without fixed points and
without invariant Baker domains.

R\"uckert and Schleicher~\cite{RucSch} have shown that if $f$ is a polynomial
and if $U$ is the immediate basin of a zero, then each 
component of $\C\setminus U$ contains the basin of another zero. They 
deduce this result from a more general result dealing with the case
that $f$ is entire but not necessarily  a polynomial. To state
this result, let again $U$ be 
the immediate basin of a zero $\xi$ of $f$ 
and suppose that there
are two $N_f$-invariant curves $\Gamma_1$ and $\Gamma_2$ which connect
$\xi$ to $\infty$ in $U\cup \{\infty\}$, which intersect only in $\xi$
and $\infty$ and which are not homotopic (with fixed endpoints) in
$U\cup \{\infty\}$. Let $\widetilde{V}$ be a component of 
$\C\setminus \left(\Gamma_1 \cup \Gamma_2\right)$. 
With these notations their main 
result~\cite[Theorem~5.1]{RucSch} takes the following form.
\begin{thrs}
If no point in $\CC$ has infinitely many preimages within~$\widetilde{V}$,
then the set $V:=\widetilde{V}\setminus U$ contains an immediate basin
or a virtual immediate basin of~$N_f$.
\end{thrs}
R\"uckert and Schleicher raise 
the question whether the hypothesis that 
no point in $\CC$ has infinitely many preimages within $\widetilde{V}$
is necessary. We show that this is indeed the case.
\begin{thm}
\label{th2}
There exists an entire function $g$ with exactly one zero at $0$
such that the immediate basin of $0$ contains~$\R$,
but $N_g$ has no virtual immediate basin.
\end{thm}
The functions $f$ and $g$ in Theorems~\ref{th1} and~\ref{th2}
can be given explicitly. Let $(r_k)$ be a sequence of real numbers
tending to $\infty$ and let $(n_k)$ be a sequence of positive
integers satisfying $n_k\geq k$ 
for all $k\in\N$.
Then 
\begin{equation}
\label{defh}
h(z):=\prod_{k=1}^\infty \left(1+\left(\frac{z}{r_k}\right)^{n_k}
\right)
\end{equation}
defines an entire function~$h$.
We shall show that if 
\begin{equation}
\label{conditions}
r_k\geq 2 r_{k-1} \geq 2,\quad 
n_k\geq \sum_{j=1}^{k-1} n_j\quad  \text{and}\quad
n_k\geq r_k^{4n_{k-1}}
\end{equation}
for $k\geq 2$,
then the functions
\begin{equation}
\label{deff}
f(z):= \exp\left( \int_0^z h(t)dt\right)
\end{equation}
and 
\begin{equation}
\label{defg}
g(z):= z  \exp\left( \int_0^z \frac{h(t)-1}{t}dt\right)
\end{equation}
satisfy the conclusions of Theorems~\ref{th1} and~\ref{th2},
respectively.
\begin{ack}
I thank Xavier Buff, 
Nuria Fagella,
Lasse Rempe, Phil Rippon,
Johannes R\"uckert and Dierk Schleicher for useful
discussions on the topics of this paper.
\end{ack}
\section{Proofs of Theorem~\ref{th1} and~\ref{th2}}
We denote the open disk of radius $r$ around a point $a\in\C$ by
$D(a,r)$.
The hyperbolic metric in a plane domain $U$ is denoted by~$\lambda_U$.
By $\dist_U(z,A)$ we denote
the hyperbolic distance between a point  $z$ and a set~$A$.
We shall make use of the fact that if $A\subset U$ and if
$z_0$ is a point which is in the boundary of $U$ but not in the
closure of~$A$, then 
$\dist_U(z,A)\to \infty$ as $z\to z_0$, $z\in U$. In particular, we have
\begin{equation}
\label{hypmet}
\lim_{z\to \infty}
\dist_{\C\setminus\{0,1\}}
\left(z,D\left(\tfrac12,\tfrac12\right)\right)=\infty.
\end{equation}

To estimate the growth of $h$ we note that~\eqref{conditions} implies 
that
if  $k\geq 3$, then
$r_k\geq 2^{k-1}\geq 4$ and 
$n_{k-1}\geq r_{k-1}^{4n_{k-2}}\geq 2^{(k-1)4n_{k-2}}\geq  k+1$.
For $|z|=r_k$ where $k\geq 3$ we thus obtain
\begin{eqnarray*}
\log |h(z)|
&\leq &
\sum_{j=1}^{k-1} \log\left(1+\left|\frac{r_k}{r_j}\right|^{n_j}\right)
+\log 2 
+\sum_{j=k+1}^\infty \log\left(1+\left|\frac{r_k}{r_j}\right|^{n_j}\right)\\
&\leq &
\sum_{j=1}^{k-1} \log\left(1+{r_k}^{n_j}\right)
+\log 2
+\sum_{j=k+1}^\infty \left|\frac{r_k}{r_j}\right|^{n_j}\\
&\leq &
\sum_{j=1}^{k-1} \log\left(2{r_k}^{n_j}\right)
+\log 2
+\sum_{j=k+1}^\infty 2^{-n_j}\\
&\leq &
k\log 2 + \left(\sum_{j=1}^{k-1} n_j \right) \log r_k +1\\
&\leq & 
(k + 2 n_{k-1} +1) \log r_k \\
&\leq & 
3 n_{k-1} \log r_k.
\end{eqnarray*}
Hence
\begin{equation}
\label{estimateh}
|h(z)|\leq r_k^{3 n_{k-1}}.
\end{equation}
for $|z|=r_k$ and $k\geq 3$.

\begin{proof}[Proof of Theorem~\ref{th1}]
Let $h$ and $f$  be defined by~\eqref{defh} and~\eqref{deff} so that
$$N_f(z)=z-\frac{1}{h(z)}.$$
Suppose that
$N_f$ has an invariant Baker domain~$U$. 
Take $z_0\in U$ and connect $z_0$ by a curve $\gamma_0$ in $U$
to $N_f(z_0)$. Then $\gamma:=\bigcup_{j=0}^\infty N_f^j(\gamma_0)$
is a curve in $U$ which connects $z_0$ to~$\infty$.
By compactness, there exists $K\geq 0$ such that 
$\lambda_U(z,N_f(z))\leq K$ for all $z\in\gamma_0$.
Since every $z\in\gamma$ has the form $z=N_f^j(\zeta)$ for some
$\zeta\in\gamma_0$ and some $j\geq 0$
and since the holomorphic self-map $N_f$ of
$U$ does not increase hyperbolic distances 
this implies that 
\begin{equation}
\label{hypgamma}
\lambda_U(z,N_f(z))\leq K\quad  \text{for} \quad z\in\gamma.
\end{equation}
For large $k$ the curve $\gamma$ intersects the
circle $\{z:|z|=r_k\}$. Let
$z_k$ be a point of intersection.
Define
$$P_k:=\left\{r_ke^{(2 \nu+1) \pi i /n_k}:0\leq \nu\leq
n_k-1\right\}.$$
The $n_k$ points of $P_k$ are zeros of $h$ and
hence poles of~$N_f$. Thus $P_k\cap U=\emptyset$ for all $k\in\N$.
For $k\geq 2$ we have $n_k\geq r_k^4\geq 16$
so that $P_k$
contains more than one point.
Let $a_k,b_k$ the points of $P_k$ which are closest to~$z_k$.
Then
\begin{equation} \label{akbk}
\left|a_k-b_k\right|
=\left|e^{ 2 \pi i/n_k}-1\right|
\leq \frac{4\pi}{n_k}
\end{equation}
and 
\begin{equation} \label{akbk2}
z_k\in D\left( \tfrac12(a_k+b_k), \tfrac12|a_k-b_k|\right).
\end{equation}
Define $L_k:\C\setminus\{a_k,b_k\}\to \C\setminus\{0,1\}$ by
$L_k(z)=(z-a_k)/(b_k-a_k)$.
Then
\begin{equation} \label{hypmet2}
\lambda_{\C\setminus\{0,1\}}(L_k(z_k),L_k(N_f(z_k))) 
=\lambda_{\C\setminus\{a_k,b_k\}}(z_k,N_f(z_k))
\leq \lambda_U(z_k,N_f(z_k))
\leq K
\end{equation}
by~\eqref{hypgamma}.
By~\eqref{akbk2} we have
$L_k(z_k)\in D\left(\tfrac12,\tfrac12\right)$.
On the other hand,~\eqref{estimateh},~\eqref{akbk} and~\eqref{conditions} 
imply that
\begin{eqnarray*}
|L_k(N_f(z_k))|
&\geq &
|L_k(N_f(z_k))-L_k(z_k)|-|L_k(z_k)|\\
&=&
\frac{|N_f(z_k)-z_k|}{|a_k-b_k|}-|L_k(z_k)|\\
&=&
\frac{1}{|h(z_k) (a_k-b_k)|}-|L_k(z_k)|\\
&\geq& 
\frac{n_k}{4\pi r_k^{3n_{k-1}}} - 1 \\
&\geq&
\frac{r_k^{n_{k-1}}}{4\pi}-1
\end{eqnarray*}
and thus $|L_k(N_f(z_k))|\to\infty$ as $k\to\infty$.
Combining this with~\eqref{hypmet} we see that
$\lambda_{\C\setminus\{0,1\}}(L_k(z_k),L_k(N_f(z_k)))\to\infty$
as $k\to\infty$,
contradicting~\eqref{hypmet2}.
\end{proof}

\begin{proof}[Proof of Theorem~\ref{th2}]
Let $h$ and $g$  be defined by~\eqref{defh} and~\eqref{defg} so that
$$N_g(z)=z-\frac{z}{h(z)}=z\left(1-\frac{1}{h(z)}\right).$$ 
The proof that $N_g$ has no Baker domains proceeds exactly as 
the proof of Theorem~\ref{th1}.
(We only obtain 
$$|L_k(N_f(z_k))|\geq  \frac{r_k^{n_{k-1}-1}}{4\pi} -1,$$ 
but this still gives a contradiction.)

Since $h$ is real on the real axis, the same holds for~$N_g$,
and since $h(x)> 1$ for all $x\in\R\setminus\{0\}$ we see 
that $|N_g(x)|<|x|$ for all $x\in\R\setminus\{0\}$.
This implies that $N_g^k(x)\to 0$ as $k\to\infty$, for all $x\in\R$.
Hence $\R$ is contained in the immediate basin of~$0$.
\end{proof}
\section{Remarks}
\subsection*{1.} 
It follows from the result of Buff and R\"uckert~\cite{BufRuc}
already mentioned
in the introduction that the function $f$ of Theorem~\ref{th1} 
has no logarithmic singularity over~$0$. Another example of an
entire function without zeros and with no logarithmic singularity over~$0$
was given in~\cite{BerEre}.

\subsection*{2.} 
The invariant components of the Fatou set of a meromorphic function
can be classified; see~\cite{Ber93}. For functions without fixed 
points there are only two possible types of invariant components:
Baker domains and Herman rings. Fagella, Jarque and Taixes~\cite{FagJarTaiII}
have shown that a meromorphic function without fixed points does
not have Herman rings. This implies that for a function $f$ satisfying
the conclusion of Theorem~\ref{th1} the Fatou set of $N_f$
has no invariant component at all. Probably there also exist
entire functions $f$ for which the Fatou set of $N_f$ is empty.

\subsection*{3.} 
It was shown by Przyztycki \cite{Prz} that if $f$ is a poynomial,
then the immediate basin of each zero is simply connected.
Shishikura~\cite{Shi90} showed that in fact the Julia set of $N_f$ is connected;
that is, all Fatou components of $N_f$ are simply connected.
It is not known whether this last result also holds if $f$ is
an entire transcendental function, but Mayer and 
Schleicher~\cite{MaySch} have shown that immediate basins are
simply connected. Fagella, Jarque and Taixes~\cite{FagJarTaiI,FagJarTaiII}
have extended this result by showing that
immediate attracting and parabolic basins (of any period) are 
simply connected and that preimages of simply connected Fatou components
of $N_f$ are simply connected.
However, it remains open whether invariant Baker domains of 
$N_f$ are necessarily simply connected. If this is true, then
our definition of virtual immediate basins coincides with that
given in~\cite{BufRuc,RucSch} since then the additional condition
on the existence of an absorbing set is always satisfied; cf.\ the discussion
in~\cite{BufRuc,RucSch}.


\begin{thebibliography}{99}

\bibitem{Ber93}
W.\ Bergweiler,
Iteration of meromorphic functions.
{\rm Bull.\ Amer.\ Math.\ Soc.\ (N.\ S.)} 29 (1993), 151--188.
\bibitem{BerDraLan}
W.\ Bergweiler, D.\ Drasin and J.\ K.\ Langley,
Baker domains for Newton's method,
Ann.\ Inst.\ Fourier 57 (2007), 803--814.
\bibitem{BerEre}
W.\ Bergweiler and A.\ Eremenko,
Direct singularities and completely invariant domains
of entire functions,
Illinois Math.\ J., to appear.
\bibitem{BufRuc}
X.\ Buff and J.\ R\"uckert,
Virtual immediate basins of Newton maps and asymptotic values,
Int.\ Math.\ Res.\ Not.\  2006, 65498, 1--18.
\bibitem{FagJarTaiI}
N.\ Fagella, X.\ Jarque and J.\ Taixes, 
Connectivity of Julia sets of transcendental meromorphic maps
and weakly repelling fixed points (I), preprint.
\bibitem{FagJarTaiII}
N.\ Fagella, X.\ Jarque and J.\ Taixes, 
Connectivity of Julia sets of transcendental meromorphic maps
and weakly repelling fixed points (II), in preparation.
\bibitem{MaySch}
S.\ Mayer and D.\ Schleicher,
Immediate and virtual basins of Newton's method for entire functions,
Ann.\ Inst.\ Fourier 56 (2006), 325--336.
\bibitem{Ne}
R.\ Nevanlinna,
{\rm Eindeutige analytische Funktionen},
Springer, Berlin, G\"ottingen, Heidelberg, 1953.
\bibitem{Prz}
F.\ Prztycki,
Remarks on simple-connectedness of basins of sinks for 
iterations of rational maps,
in ``Dynamical Systems and Ergodic Theory,''
Banach Center Publ.\ 23 (1989), 229--235.
\bibitem{RucSch}
J.\ R\"uckert and D.\ Schleicher,
On Newton's method for entire functions,
J. London Math.\ Soc.\ 75 (2007), 659--676.
\bibitem{Shi90} M.\ Shishikura,
The connectivity of the Julia set and fixed point,
Preprint IHES/M/90/37, Institut des Hautes Etudes Scientifiques, 1990.
\end{thebibliography}
\end{document}